\documentclass[a4paper,11pt]{amsart}
\usepackage{bm}
\usepackage{amsmath,amssymb}
\usepackage{graphicx, psfrag}

\newtheorem{theorem}{Theorem}
\newtheorem{lemma}[theorem]{Lemma}
\newtheorem{proposition}[theorem]{Proposition}

\theoremstyle{definition}
\newtheorem{definition}[theorem]{Definition}

\theoremstyle{remark}
\newtheorem{remark}[theorem]{Remark}

\begin{document}
\title{On solutions to Walcher's extended holomorphic anomaly equation}
\author{Yukiko Konishi and Satoshi Minabe}
\address{Graduate School of Mathematical Sciences, The University of Tokyo,
 3-8-1 Komaba, Meguro, Tokyo 153-8914 Japan}
\email{konishi@ms.u-tokyo.ac.jp}
\address{Department of Mathematics, Hokkaido University,
Kita 10, Nishi 8, Kita-Ku, Sapporo 060-0810, Japan}
\email{minabe@math.sci.hokudai.ac.jp}

\subjclass[2000]{Primary 14J32;  Secondary 14N35, 14J81}

\maketitle
\begin{abstract}
We give a generalization of Yamaguchi--Yau's result
to Walcher's extended holomorphic anomaly equation.
\end{abstract}

\newcommand{\F}{\mathcal{F}}
\newcommand{\C}{{C_{zzz}}}
\newcommand{\G}{{G_{z\bar{z}}}}
\newcommand{\Gam}{\Gamma_{zz}^z}
\newcommand{\Ct}{{C_{\bar{z}}^{zz}}}
\newcommand{\del}{\partial_z}
\newcommand{\delbar}{\partial_{\bar{z}}}
\newcommand{\mcpl}{\mathcal{M}_{cpl}}
%%%%%%%%%%%%%%%%%%%%%%%%%%%%%%%
\section{Introduction} 

Let $X$ be a nonsingular quintic hypersurface in $\mathbb{CP}^4$. 
The case of the $X$ and its mirror is the most well-studied example of
the mirror symmetry.
After the construction of the mirror family of 
Calabi--Yau threefolds \cite{GP},
the genus zero Gromov--Witten (GW) potential of $X$
were computed via 
the Yukawa coupling of the mirror family \cite{CDGP}.
The predicted mirror formula was proved first by Givental \cite{Givental}. 

For higher genera, Bershadsky--Cecotti--Ooguri--Vafa (BCOV) \cite{BCOV}
has predicted that the GW potential at genus $g$ is obtained as a certain limit 
of the B-model closed topological string amplitude 
$\mathcal{F}^{(g)}$ of genus $g$
\footnote{
For genus $g=0$,
the third covariant derivative of $\mathcal{F}^{(0)}$
is the Yukawa coupling, 
and for $g=1$, it is recently proved that 
$\mathcal{F}^{(1)}$ is the Quillen's norm function
\cite{FZY}.
For genus $g\geq 2$, the mathematical definition of 
$\mathcal{F}^{(g)}$ is yet to be known.}.  
They have also proposed a partial differential equation (PDE) for $\mathcal{F}^{(g)}$, 
called the BCOV holomorphic anomaly equation, which determines 
$\mathcal{F}^{(g)}$ up to a holomorphic function.
The prediction of BCOV for the genus one GW potential  was proved by 
Zinger \cite{Zinger}. 

Recently the open string analogue of 
the mirror symmetry 
has been
developed by Walcher \cite{Wa1} 
for the pair $(X,L)$ of
the quintic $3$-fold $X$ defined over $\Bbb{R}$
(called a real quintic) and the set of real points $L=X(\Bbb{R})$
which is a Lagrangian submanifold of $X$. 
Open mirror symmetry gave the prediction for the generating function for the disc GW invariants 
of $X$ with boundary in $L$ and it was
proved by Pandharipande--Solomon--Walcher \cite{PaSoWa}. 
Then, Walcher \cite{Wa2} further proposed
the open string analogue of BCOV,   
the extended holomorphic anomaly equation,
which is a PDE  for the B-model topological string amplitude $\mathcal{F}^{(g,h)}$
for world-sheets with $g$ handles and $h$ boundaries
\footnote{There is also a proposal by Bonelli--Tanzini \cite{BT}.
}.

At present there are two ways to solve the BCOV
holomorphic anomaly equation.
The one is to repeatedly use the 
identity called
the special geometry relation, or
equivalently to
draw Feynman diagrams associated to 
the perturbative expansion of a certain path
integral \cite{BCOV}.
The other is 
to solve the system of PDE's due to Yamaguchi--Yau \cite{YY}.
They showed that
$\mathcal{F}^{(g)}$ multiplied by $(g-1)$-th powers of the 
Yukawa coupling, is a polynomial in 
finite number of generators
and rewrite the holomorphic anomaly equation as
PDE's with respect to these generators.
This result were then reformulated into a more useful form by 
Hosono--Konishi \cite[\S 3.4]{HoKo}.

It is a natural problem to generalize  these methods 
to Walcher's extended holomorphic anomaly equation. 
The generalization of the Feynman rule method  
can be obtained from the result of  Cook--Ooguri--Yang \cite{COY}.
The objective of this article is to 
generalize Yamaguchi--Yau's and Hosono--Konishi's results to the
extended holomorphic anomaly equation.
It gives more tractable method in computations than the one given by the Feynman rule. 

The organization of the paper is as follows.
In Section 2, we recall the special K\"ahler geometry
of the B-model complex moduli space and
Walcher's extended holomorphic anomaly equation.
We also describe the Feynman rule.
In Section 3, we
rewrite the holomorphic anomaly equation as
PDE's (Theorem \ref{prop:EHAE2}). 
In Section 4, we compute several BPS numbers
by fixing holomorphic ambiguities 
with certain assumptions. 
The assumptions 
in this section are experimental in a sense. 
In appendices we include
the Feynman diagrams and the solution of the PDE's 
for $(g,h)=(0,4)$.

After we finished writing this paper,
we were informed that Alim--L\"ange \cite{AlLa} 
also obtained a generalization of Yamaguchi--Yau's result.

\subsection*{Acknowledgments}

Y.K. thanks Shinobu Hosono 
for valuable discussions and helpful comments. 
This work was initiated when S.M. was staying at 
Institut Mittag-Leffler (Djursholm, Sweden).
He would like to thank the institute for support. 
The authors are also grateful to J.D. L\"ange and T. Okuda for 
informing them of the work of \cite{AlLa} at
the 5th Simons Workshop on Mathematics and Physics held at Stony Brook. 
The work of Y.K. is partly supported by 
JSPS Research Fellowships
for Young Scientists.
Research of S.M. is supported in part  by 
21st Century COE Program at Department of Mathematics, Hokkaido University.   
%%%%%%%%%%%%%%%%%%%%%%%%%%%%
\section{Walcher's extended holomorphic anomaly equation}
%%%%%%%%%%%%%%%%%%%%%%%%%%%%

%%%%%%%%%%%%%%%%%%%%%%%%%%%%%%%%%%%%%
\subsection{Special K\"ahler geometry} 
%%%%%%%%%%%%%%%%%%%%%%%%%%%%%%%%%%%%%
Recall the mirror family of the quintic hypersurface
$X\subset \mathbb{P}^4$
constructed in \cite{GP}. 
Let $W_{\psi}$ be the hypersurface in $\Bbb{P}^4$ defined by 
\begin{equation*}
\sum_{i=0}^{4} x_i^5 -5\psi \prod_{i=0}^4 x_i=0.
\end{equation*}
After taking the quotient by $(\Bbb{Z}/5\Bbb{Z})^3$ and  
a crepant resolution $Y_{\psi}$ of $W_{\psi}/ (\Bbb{Z}/5\Bbb{Z})^3$, 
we obtain a one-parameter family of Calabi--Yau threefolds 
$\pi :\mathcal{Y}\to \mcpl  :=\mathbb{P}^1\setminus\{0,\frac{1}{5^5},\infty\}$, 
where a local coordinate $z$ of $\mcpl$ is given by
$z=(5\psi)^{-5}$.

Consider the variation of Hodge structure of weight three 
on the middle cohomology groups $H^3(Y_z, \Bbb{C})$. 
Let $0\subset F^3 \subset F^2 \subset F^1 \subset F^0=R^3\pi_*\Bbb{C}\otimes \mathcal{O}_{\mcpl}$ 
be the Hodge filtration 
and $\nabla$ be the Gauss--Manin connection. 
The holomorphic line bundle $\mathcal{L}:=F^3$ over $\mcpl$  
is called the vacuum line bundle 
(the fiber of $\mathcal{L}$ at $z$ is $H^{3,0}(Y_z)$). 
Let $\Omega(z)$ be a local holomorphic section trivializing $\mathcal{L}$, i.e.
a nowhere vanishing $(3,0)$-form on $Y_{z}$. 
The Yukawa coupling $\C$ is define by 
$$\C:=\int_{X_z} \Omega(z)\wedge ({\nabla_{\del}})^3 \Omega(z), $$ 
which is a holomorphic section of ${\rm Sym}^3(T_{\mcpl}^*) \otimes (\mathcal{L}^*)^2$, where 
$T_{\mcpl}^*$ denotes the holomorphic cotangent bundle of $\mcpl$. 
A suitable choice of  $\Omega(z)$  gives (\cite{CDGP})
\begin{equation}\nonumber
C_{zzz}=\frac{5}{(1-5^5z)z^3}.
\end{equation}
It also gives the following 
Picard--Fuchs operator $\mathcal{D}$
which governs the periods of $\Omega(z)$ : 
$$
\mathcal{D}=\theta_z^4-5z(\theta_z+1)(\theta_z+2)
(\theta_z+3)(\theta_z+4),
$$
where $\theta_z=z\frac{d}{dz}$.

Consider the pairing
\begin{equation*}
(\phi, \psi) := \sqrt{-1} \int_{Y_z} \phi \wedge {\psi}, \quad \phi, \psi \in H^{3}(Y_z,\mathbb{C}). 
\end{equation*}
Then $(~~,\overline{\phantom{O}})$ induces a Hermitian metric on $\mathcal{L}$. 
Let $K(z, \bar{z}):=-\log(\Omega(z),\overline{\Omega(z)})$. 
This defines a K\"{a}hler metric (the Weil-Peterson metric)  $\G:=\del\delbar K$  on $\mcpl$. 
There is a unique holomorphic Hermitian connection $D$ on 
$(T_{\mcpl})^{m}\otimes\mathcal{L}^{n}$ whose $(1,0)$-part $D_z$ is given by
\begin{equation*}
D_z =\del +m \Gam+ n(-\del K)~,
\end{equation*}
where  $\Gam=G^{z\bar{z}}\del G_{z\bar{z}}$. 
An important property of   
$\G$ is the following identity called the special geometry relation \cite{Strominger}
\begin{equation}\label{eq:special}
\delbar \Gam = 2G_{z\bar{z}}-\C C_{\bar{z}\bar{z}\bar{z}}e^{2K}
G^{z\bar{z}}G^{z\bar{z}}, 
\end{equation}
where 
$C_{\bar{z}\bar{z}\bar{z}}:=\overline{\C}$.

Now we introduce
the open disk amplitude with two insertions $\triangle_{zz}$,
which is 
the open-sector analogue of the Yukawa coupling.
Let $\mathcal{T}$ be a holomorphic section of $\mathcal{L}^*$ locally given by 
\begin{equation}\label{eq:tau}
  \mathcal{T} = 60~\tau(z),~
\qquad
 \tau(z) =\sum_{n=0}^{\infty}
  \frac{(\frac{7}{2})_{5n}}{({(\frac{3}{2})_n})^5} ~
    z^{n+\frac{1}{2}}. 
\end{equation}
Here $(\alpha)_n$ is the Pochhammer symbol : 
$(\alpha)_n:=\alpha(\alpha+1)\cdots (\alpha+n-1)$ for $n>0$ and $(\alpha)_0:=1$. 
$\mathcal{T}$ is a solution to 
\begin{equation}\label{eq:EXPF}
\mathcal{D}\mathcal{T}=\frac{60}{2^4}\sqrt{z}.
\end{equation} 
Following \cite{Wa2}, we define  
a $C^{\infty}$--section  $\triangle_{zz}$ of ${\rm Sym}^2(T_{\mcpl}^*) \otimes \mathcal{L}^*$ 
by
\begin{equation}
\triangle_{zz}=D_z D_z\mathcal{T}
-\frac{e^K C_{zzz}}{G_{z\bar{z}}}\overline{D}_{\bar{z}}\overline{\mathcal{T}}~, 
\end{equation}
where $\overline{D}_{\bar{z}}=\delbar+\delbar K$ denotes the $(0,1)$-part of $\overline{D}$.  
By (\ref{eq:special}), it follows  that $\triangle_{zz}$ satisfies 
the equation
\begin{equation}\label{eq:hadisc}
\delbar \triangle_{zz} = -\C e^{K} G^{z\bar{z}} \triangle_{\bar{z}\bar{z}}, 
\end{equation}
where $\triangle_{\bar{z}\bar{z}}:=\overline{\triangle_{zz}}$. 

\begin{remark}
In \cite{Wa2}, it is argued that 
$\mathcal{T}$ and $\triangle_{zz}$ should be written as 
\begin{eqnarray*}
\mathcal{T}(z)=\int_{Y_z} \Omega(z) \wedge \tilde{\nu}(z), \quad
\triangle_{zz}=\int _{Y_z} \Omega(z) \wedge \nabla^2 \tilde \nu(z),
\end{eqnarray*} 
where $\tilde{\nu}$ is a $C^{\infty}$--section of the Hodge bundle $F^0$ 
which is the `real horizontal lift' of 
a certain Griffiths normal function $\nu$ 
associated to a family of homologically trivial 2-cycles
\footnote{
By definition,  $\nu$ is a holomorphic and horizontal section of 
the intermediate Jacobian fibration $\mathcal{J}^3 \to \mcpl$ of $\mathcal{Y} \to \mcpl$. 
See, e.g., \cite{Green, Griffiths}.}. 
The normal function $\nu$ should be determined from the Lagrangian submanifold $L\subset X$ 
under the mirror symmetry with D-branes.   
\end{remark}

%%%%%%%%%%%%%%%%%%%%%%%%%%%%%%%%%%%%%%%%%%%%%%%%%%
\subsection{Extended holomorphic anomaly equation}
%%%%%%%%%%%%%%%%%%%%%%%%%%%%%%%%%%%%%%%%%%%%%%%%%%
Let $\F^{(g,h)}$ be the B-model topological string amplitude
of genus $g$ with $h$ boundaries,
and let
$$
\F^{(g,h)}_0:=\F^{(g,h)},\qquad
\F^{(g,h)}_n:=D_z \F^{(g,h)}_{n-1}~~(n\geq 1). 
$$
$\F^{(g,h)}_n$ is a 
$C^{\infty}$--section of the line bundle 
$(T_{\mcpl}^*)^n\otimes \mathcal{L}^{2g-2+h}$.
For $(g,h)=(0,0),(0,1)$,
\begin{equation}
 \F^{(0,0)}_3=C_{zzz},\quad
 \F^{(0,1)}_2=\triangle_{zz}.
\end{equation}
For $(g,h)=(1,0),(0,2)$%
%%%%%%%%%%%
 \footnote{
$\F^{(1,0)}_1$ and $\F^{(0,2)}_1$ are solutions to
the following (extended) holomorphic anomaly equations 
\cite{BCOV}\cite{Wa2}. 
\begin{equation}\nonumber
 \delbar\F^{(1,0)}_1
  =\frac{1}{2}\C\Ct-
     \Big(\frac{\chi}{24}-1\Big)\G
,\quad
  \delbar\F^{(0,2)}_1=
   -\triangle_{zz}\triangle_{\bar{z}}^z+
    \frac{N}{2}\G.
\end{equation}
}%
,
\begin{equation}
\begin{split}
 \F^{(1,0)}_1&=\frac{1}{2}\del \log\Big(
   e^{(4-\frac{\chi}{12})K} \G^{-1} 
       (1-5^5z)^{-\frac{1}{6}} z^{-1-\frac{c_2\cdot H}{12}}
 \Big),
\\
 \F^{(0,2)}_1&=-\triangle_{zz}\triangle^z
  -\frac{1}{2}\C \triangle^z\triangle^z +\frac{N}{2}\del K
 +f^{(0,2)},\quad 
f^{(0,2)}=\frac{75}{2(1-5^5z)}~,
\end{split}
\end{equation}
where 
  $\chi=-200$,
  $c_2\cdot H=50$, 
  $N=1$ and
 $\triangle^z = -\frac{\triangle_{zz}}{\C}$
(cf. \S \ref{sec:propagators}).

As in \cite{Wa2}, define
$$
\Ct = C_{\bar{z}\bar{ z} \bar{z}}e^{2K} \G^{-2}~,
\qquad
\triangle_{\bar{z}}^z = \triangle_{\bar{z}\bar{z}}
                   e^K \G^{-1}~.
$$
Then Walcher's extended holomorphic anomaly equation for
$(g,h)\neq (0,0),(1,0),(0,1),(0,2)$ is
as follows.
\begin{equation}\label{eq:EHAE}
\begin{split}
   \delbar\F^{(g,h)} 
 &=
   \frac{1}{2}\Ct \Big(
   \sum_{g_1,g_2,h_1,h_2}
     \F^{(g_1,h_1)}_1 \F^{(g_2,h_2)}_1 + \F^{(g-1,h)}_2
  \Big)
 -\triangle_{\bar{z}}^z \mathcal{F}^{(g,h-1)}_1~.
\end{split}
\end{equation}
In the RHS, the summation is over 
$g_1,h_1,g_2,h_2\geq 0$
satisfying
 $g_1+g_2=g$,
 $h_1+h_2=h$ and
 $(g_1,h_1),(g_2,h_2)\neq (0,0),(0,1)$.
The second and the third terms in the RHS 
should be set to zero
if $g=0$ and $h=0$, respectively.

%%%%%%%%%%%%%%%%%%%%%%%%%%%%%%%%%%%%%%%%
\subsection{Propagators and Terminators}\label{sec:propagators}
%%%%%%%%%%%%%%%%%%%%%%%%%%%%%%%%%%%%%%%%
We introduce 
the propagators 
$S^{zz},S^z,S$ and the terminators 
$\triangle^z,\triangle$ \cite{BCOV,Wa2}. 
By definition, they are solutions to
\begin{equation}
\begin{split}
 \label{eq:delbarS}
 & \delbar S^{zz}=\Ct,   \quad
   \delbar S^z=S^{zz}\G, \quad
   \delbar S=S^z \G,
\\
 & \delbar \triangle^z =\triangle_{\bar{z}}^{z},
   \qquad\qquad
   \delbar \triangle=\triangle^z \G.
\end{split}
\end{equation}
These equation can be solved by using (\ref{eq:special}) and (\ref{eq:hadisc}). 
The solutions of the propagators are \cite[p.391]{BCOV}.
\begin{equation}\label{eq:propagators0}
\begin{split}
 S^{zz} &= 
  \frac{1}{C_{zzz}}
     \big( 2\partial_{z}\log(e^K |f|^2)-
           \partial_{z}\log (v G_{z\bar{z}}) 
      \big)~,
\\
  S^{z} &= \frac{1}{C_{zzz}}
     \big((\partial_{z}\log(e^K|f|)^2 -
      v^{-1}\partial_{z}v\partial_z 
            \log(e^K|f|^2)  
     \big) ~,
\\
  S &= \big(S^{z}-\frac{1}{2}D_{z}S^{zz}
     - \frac{1}{2}(S^{zz})^2C_{zzz} \big)
      \partial_{z}\log (e^{K}|f|^2)
     + \frac{1}{2}D_{z}S^{z}
     + \frac{1}{2}S^{zz}S^{z}C_{zzz}
   ~.
\end{split}
\end{equation}
Here
$f,v$ are holomorphic functions of $z$.
We take $f=z^{-\frac{1}{5}}$ and $v=\frac{dz}{d\psi}$
($z=\frac{1}{5^5\psi^5}$) so that $S^{zz},S^z,S$
do not diverge at $z=\infty$%
\footnote{If rewritten in  the $\psi$-coordinate,
\eqref{eq:propagators0} are the same as
those used in \cite[3.11]{Wa2}\cite[(2.21)]{YY}.}%
.
Solutions of the terminators are \cite[(3.12)]{Wa2}
\begin{equation}\label{eq:terminators}
 \triangle^z=-\frac{\triangle_{zz}}{C_{zzz}},\quad
\triangle=D_z \triangle^z.
\end{equation}

%%%%%%%%%%%%%%%%%%%%%%%%%
\subsection{Feynman Rule}
%%%%%%%%%%%%%%%%%%%%%%%%%
We describe the Feynman 
rule which gives a solution to \eqref{eq:EHAE}.

For non-negative integers $g$,$h$,$m$, and $n$,  
we define  $\widetilde{C}^{(g,h)}_{{n}: {m}}$ 
recursively as follows. 
\begin{eqnarray}
\widetilde{C}^{(0,0)}_{0: m}=\widetilde{C}^{(0,0)}_{1: m}
=\widetilde{C}^{(0,0)}_{2: m} = 0, \\
\widetilde{C}^{(0,1)}_{0: m} = \widetilde{C}^{(0,1)}_{1:m}=0,\\
\widetilde{C}^{(0,2)}_{0: 1}=-\frac{N}{2} ,\\
\widetilde{C}^{(1,0)}_{0: 0}=0, \quad
\widetilde{C}^{(1,0)}_{0: 1}=\frac{\chi}{24}-1,\\
{C}^{(g,h)}_{n}= \mathcal{F}^{(g,h)}_n \quad \mathrm{if} \quad 2g-2+h+n \geq1, \\
\widetilde{C}^{(g,h)}_{n: 0}={C}^{(g,h)}_{n}, \quad
\mathrm{if}  \quad 2g-2+h+n\geq 1,\\
\widetilde{C}^{(g,h)}_{n: {m+1}}=(2g-2+h+n+m) \widetilde{C}^{(g,h)}_{n: m}.
\end{eqnarray} 

\begin{definition}
A Feynman diagram $G$ is a finite labeled graph 
$$G=(V; E_0^{\rm in}, E_1^{\rm in}, E_2^{\rm in}, E_0^{\rm out}, E_1^{\rm out}; j),$$ 
which consists of the following data. \\
(i)  
Each vertex $v\in V$ is labeled by a pair of non-negative integers $(g_v, h_v)$.\\
(ii) There are three kinds of inner edges 
$E^{\rm in}=E_0^{\rm in} \sqcup E_1^{\rm in} \sqcup E_2^{\rm in}$  
and two kinds of outer edges $E^{\rm out} = E_0^{\rm out} \sqcup E_1^{\rm out}$. 
The end points of the edges are specified by 
the collection of maps $j=(j^{\rm in}_0,  j^{\rm in}_1, j^{\rm in}_2, j^{\rm out}_0, j^{\rm out}_1)$  :
\begin{eqnarray*} 
&j^{\rm in}_0: E_0^{\rm in} \to (V \times V)/\sigma,  \quad
j^{\rm in}_1: E_1^{\rm in} \to V \times V, \quad 
j^{\rm in}_2: E_2^{\rm in} \to (V \times V)/\sigma, \\
&j^{\rm out}_0:E^{\rm out}_0 \to V,\quad
j^{\rm out}_1:E^{\rm out}_0 \to V, 
\end{eqnarray*}
where $\sigma: V \times V \to V \times V$ is the involution 
interchanging the first and the second factors. 
\end{definition}
In a more plain language,
an edge of type $E_i^{in}$ has both endpoints attached to vertices,
and an edge of type $E_{i}^{out}$ has only one endpoint attached to
a vertex.
We represent edges of types $E_{0}^{in}$ and $E_0^{out}$
by solid lines, edges of types $E_{2}^{in}$ and $E_1^{out}$ by dashed lines
and an edge of type $E_1^{in}$ by a half-solid, half-dashed line.
See Fig. \ref{fig:edges}. 
 
For a vertex $v\in V$,  we set 
\begin{eqnarray*}
&L_{i,v}=\{e\in E_i^{\rm in} \mid j^{\rm in}_i(e) = \{ v, v\} \}, 
\quad L_i = \bigsqcup_{v\in V} L_{i,v}, 
 \; \;(i=0,2),   
\\
&L_{1,v}=\{e \in E_1^{\rm in} \mid j^{\rm in}_1(e)= (v,v )\}.
\end{eqnarray*}
In other words, $L_{i,v}$ is the number of self-loops attached to 
the vertex $v$ whose edges are of the type $E_i^{\rm in}$.
Define non-negative integers $n_v^{\rm in}$, 
$n_v^{\rm out}$, $m_v^{\rm in}$ and $m_v^{\rm out}$ by 
\begin{eqnarray*}
&n_v^{\rm in}=\#\{  e\in E_2^{\rm in} \mid v \in j_2^{\rm in}(e)  \}
+\#\{  e\in E_1^{\rm in} \mid j_1^{\rm in}(e)=(v, \ast)  \}
+\#L_{2,v} +\#L_{1,v}, \\
&m_v^{\rm in}=\#\{  e\in E_0^{\rm in} \mid v \in j_0^{\rm in}(e)  \}
+\#\{  e\in E_1^{\rm in} \mid j_1^{\rm in}(e)=(\ast, v)  \}
+\#L_{0,v} +\#L_{1,v}, \\
&n_v^{\rm out}=\#\{  e\in E_1^{\rm out} \mid v \in j_1^{\rm out}(e)  \}, \quad 
m_v^{\rm out}=\#\{  e\in E_0^{\rm out} \mid v \in j_0^{\rm out}(e)  \}.
\end{eqnarray*}
The valence ${\rm val}(v)$ of $v \in V$ is given by 
${\rm val}(v)=n_v+m_v$, where 
$n_v:=n_v^{\rm in}+n_v^{\rm out}$ (the number of solid lines attached to $v$), 
$m_v:=m_v^{\rm in}+m_v^{\rm out}$ (the number of dashed lines attached to $v$).
See Fig. \ref{fig:vertex}. 

\begin{definition}
(i) For a Feynman diagram $G$, define 
\begin{equation}\label{eq:FG}
F_G = 
\prod_{v\in V} \widetilde{C}^{(g_v,h_v)}_
{n_v: m_v}
\cdot
\prod_{e\in E_0^{\rm in}}(-2S) 
\cdot  \prod_{e\in E_1^{\rm in}}(-S^{z})
\cdot  \prod_{e\in E_2^{\rm in}}(-S^{zz}) 
\cdot \prod_{e\in E_0^{\rm out}} \Delta 
\cdot \prod_{e\in E_1^{\rm out}} \Delta^z. 
\end{equation}
(ii) 
Let ${\rm Aut}(G)$ be the automorphism group 
of $G$. Define the group $A_G$ by 
\begin{equation*}
A_G= \prod_{e \in L_0 \sqcup L_2} \Bbb{Z}/ 2\Bbb{Z}\; \ltimes {\rm Aut} (G), 
\end{equation*}
i.e. $A_G$ fits into the following exact sequence: 
\begin{equation*}
1 \rightarrow (\Bbb{Z}/ 2\Bbb{Z})^{\#L_0 + \# L_2} \rightarrow A_G 
\rightarrow {\rm Aut} (G) \rightarrow 1. 
\end{equation*}
This means that each self-loop of type $E_0^{\rm in}$ and $E_2^{\rm in}$
contributes the factor 2 to $\#A_G$.
\end{definition}

%%%%%%%%%%%%%%%%%%%%%%%
\begin{figure}[t]
\unitlength .10cm
\begin{center}
\begin{picture}(30,30)(40,20)
\thicklines
\put(15,45){(i)}
\put(15,35){$e$}
\put(0,40){\circle*{2}}
\put(-1,36){$v_1$}
\put(0,40){\line(1,0){3}}
\put(5,40){\line(1,0){3}}
\put(10,40){\line(1,0){3}}
\put(15,40){\line(1,0){3}}
\put(20,40){\line(1,0){3}}
\put(25,40){\line(1,0){5}}
\put(30,40){\circle*{2}}
\put(29,36){$v_2$}
\put(35,40){$=-2S$}
\put(15,25){(ii)}
\put(15,15){$e$}
\put(0,20){\circle*{2}}
\put(-1,16){$v_1$}
\put(0,20){\line(1,0){15}}
\put(16,20){\line(1,0){3}}
\put(20,20){\line(1,0){3}}
\put(25,20){\line(1,0){5}}
\put(30,20){\circle*{2}}
\put(29,16){$v_2$}
\put(35,20){$=-S^z$}
\put(15,5){(iii)}
\put(15,-5){$e$}
\put(0,0){\circle*{2}}
\put(-1,-4){$v_1$}
\put(0,0){\line(1,0){30}}
\put(30,0){\circle*{2}}
\put(29,-4){$v_2$}
\put(35,0){$=-S^{zz}$}
\put(70,35){(iv)}
\put(75,25){$e$}
\put(60,30){\circle*{2}}
\put(59,26){$v$}
\put(60,30){\line(1,0){3}}
\put(65,30){\line(1,0){3}}
\put(70,30){\line(1,0){3}}
\put(75,30){\line(1,0){3}}
\put(80,30){\line(1,0){3}}
\put(85,30){\line(1,0){5}}
\put(90,30){\circle{3}}
\put(95,30){$=\Delta$}
\put(72,15){(v)}
\put(75,5){$e$}
\put(60,10){\circle*{2}}
\put(59,6){$v$}
\put(60,10){\line(1,0){30}}
\put(90,10){\circle{3}}
\put(95,10){$=\Delta^{z}$}
\end{picture}
\end{center}
\vspace{3cm}
\caption{
Three types of inner edges and propagators: 
(i) $e\in E_0^{\rm in}$, $j_0^{\rm in}(e)=\{v_1,v_2\}$, 
(ii) $e\in E_1^{\rm in}$, $j_1^{\rm in}(e)=(v_1,v_2)$, 
(iii) $e\in E_2^{\rm in}$, $j_2^{\rm in}(e)=\{v_1,v_2\}$.
Two types of outer edges and terminators:   
(iv) $e\in E_0^{\rm out}$, $j_0^{\rm out}(e)=v$,   
(v) $e\in E_1^{\rm out}$, $j_1^{\rm out}(e)=v$.}
\label{fig:edges}
\end{figure}
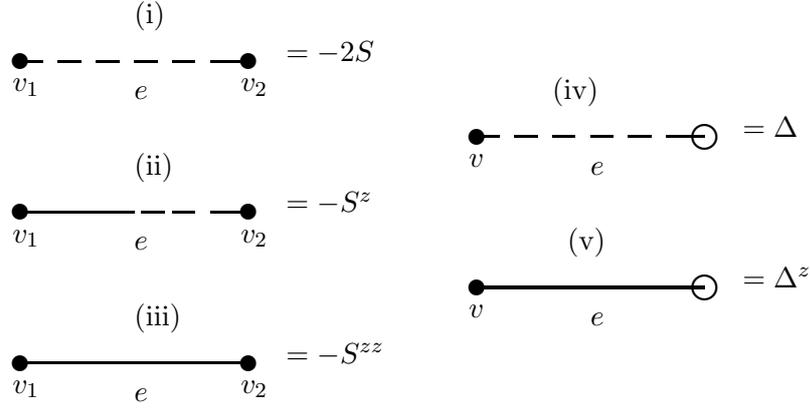
%%%%%%%%%%%%%%%%%%%%%%%%
%%%%%%%%%%%%%%%%%%%%%%%
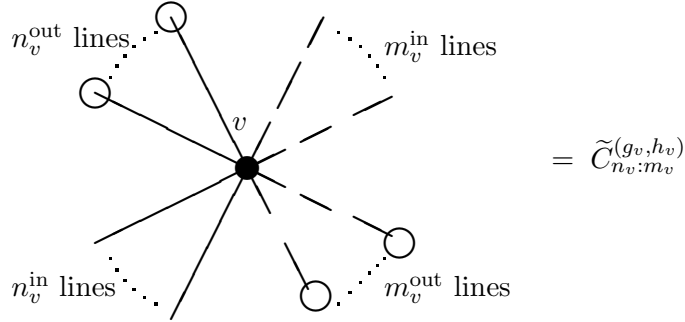
\begin{figure}[t]
\unitlength .10cm
\begin{center}
\begin{picture}(35,35)(55,25)
\thicklines
\put(48,35){$v$}
\put(50,30){\circle*{3}}
\put(50,30){\line(-1,-2){10}}
\put(50,30){\line(-2,-1){20}}
\qbezier[5](32,42)(34,46)(38,48)
\put(19,46){$n_v^{\rm out}$ lines}
\put(30,40){\circle{4}}
\put(50,30){\line(-2,1){20}}
\put(40,50){\circle{4}}
\put(50,30){\line(-1,2){10}}
\qbezier[5](32,18)(34,14)(38,12)
\put(19,13){$n_v^{\rm in}$ lines}
\put(50,30){\line(1,-2){3}}
\put(55,20){\line(1,-2){3}}
\put(50,30){\line(2,-1){5}}
\put(57,26.5){\line(2,-1){4}}
\put(64,23.5){\line(2,-1){5}}
\put(50,30){\line(2,1){5}}
\put(57,33.5){\line(2,1){4}}
\put(64,37){\line(2,1){5}}
\put(50,30){\line(1,2){3}}
\put(54,38){\line(1,2){3}}
\put(58,46){\line(1,2){2}}
\put(70,20){\circle{4}}
\put(59,13){\circle{4}}
\qbezier[5](62,48)(66,46)(68,42)
\put(68,45){$m_v^{\rm in}$ lines}
\qbezier[5](62,12)(66,16)(68,18)
\put(68,13){$m_v^{\rm out}$ lines}
\put(90,30){${\Large=\; \widetilde{C}^{(g_v,h_v)}_{n_v: m_v}}$}
\end{picture}
\end{center}
\vspace{2cm}
\caption{A vertex $v$ labeled by $(g_v, h_v)$
to which $n_v=n^{\rm in}_v + n^{\rm out}_v$ solid lines and 
$m_v=m^{\rm in}_v + m^{\rm out}_v$ dashed lines are attached and its value. }
\label{fig:vertex}
\end{figure}

\begin{definition}\label{def:graph-set}
Let $\Bbb{G}(g,h)$ be the set of (isomorphism classes 
of) Feynman diagrams $G$ which satisfy the following conditions. \\
(i)  $G$ is connected. \\
(ii) For any $v \in V$, 
$ \widetilde{C}^{(g_v,h_v)}_{{n_v}:{m_v}}
 \neq 0$. \\
(iii) $G$ satisfies 
$\sum_{v\in V}g_v + \# E^{\rm in} - \# V +1 =g$ and 
$\sum_{v\in V}h_v +\# E^{\rm out} = h$. \\
(iv) For any $v \in V$, ${\rm val}(v) > 0$.\\
\end{definition}

Note that the set $\Bbb{G}(g,h)$ is a finite set. 
Note also that the graph whose amplitude is $\mathcal{F}^{(g,h)}$,
i.e. the graph
with only one vertex with label $(g,h)$
and without edges is not a member of $\Bbb{G}(g,h)$ by (iv).

Define 
\begin{equation}\label{eq:FD}
\mathcal{F}^{(g,h)}_{\rm FD}:=-\sum_{G \in \Bbb{G}(g,h)}\frac{1}{\#A_G} F_G.
\end{equation}
The next result follows from \cite{COY}. 
\begin{proposition}
 $
  \delbar \mathcal{F}^{(g,h)}_{\rm FD}
 =\text{the RHS of }\eqref{eq:EHAE}
 $.
\end{proposition}

Therefore, the general solution 
$\mathcal{F}^{(g,h)}$ of Walcher's 
holomorphic anomaly equation is of the form 
\begin{equation}
\mathcal{F}^{(g,h)} = \mathcal{F}^{(g,h)}_{\rm{FD}} + f^{(g,h)},  
\end{equation}
where $f^{(g,h)}$ is the holomorphic ambiguity 
which can not be determined from the equation \eqref{eq:EHAE}. 
%%%%%%%%%%%%%%%%%%%%%%%%%%%%%%%%%%
\subsection{Holomorphic ambiguity}
%%%%%%%%%%%%%%%%%%%%%%%%%%%%%%%%%%
Recall that
the holomorphic ambiguity $f^{(g,0)}$ ($g\geq 2$)
is of the form \cite{BCOV}\cite[(2.30)]{YY}
$$
f^{(g,0)}=\frac{a_0 +a_1 z+\cdots+a_{2g-1}z^{2g-1}}{(1-5^5z)^{2g-2}}
+\sum_{i=0}^{\lfloor \frac{2g-2}{5}\rfloor} z^j
$$
for the closed sector $h=0$. 
Huang--Klemm--Quackenbush  \cite{HKQ} determined the holomorphic
ambiguity up to
$g\leq 51$ by using the 
vanishing of the BPS numbers $n^{g}_d$ (cf. footnote \ref{ft:GoVa}),
the gap condition at the conifold point 
$z=\frac{1}{5^5}$
and the regularity condition at the orbifold point
$z=\infty$.

For $h>0$, 
we assume that 
$\F^{(g,h)}$ has poles of order at most $2g-2+h$ at $z=\frac{1}{5^5}$
and also that the asymptotic 
behaviour at $z=\infty$ is 
$F^{(g,h)}\sim z^{\frac{2g-2+h}{2}}$ \cite[\S 3.3]{Wa2}.
Therefore we put the following ansatz for
$f^{(g,h)}$:
\begin{equation}\label{eq:hol-amb}
\begin{split}
f^{(g,h)}&=
\frac{
a_0+a_1 z+\cdots+a_{3g-3+\frac{3h}{2}} z^{3g-3+\frac{3h}{2}}
}
{(1-5^5z)^{2g-2+h}}
\quad (h \text{ even}), 
\\
f^{(g,h)}&= \frac{\sqrt{z}(
a_0+a_1 z+\cdots+
a_{3g-3+\frac{3h-1}{2}} 
z^{3g-3+\frac{3h-1}{2}})
}
{(1-5^5z)^{2g-2+h}}
\quad (h \text{ odd}).
\end{split}
\end{equation}

%%%%%%%%%%%%%%%%%%%%%%%%%%%%%%%%%%%%%%%%%%%%
\section{Polynomiality and PDE's for $\F^{(g,h)}$}
%%%%%%%%%%%%%%%%%%%%%%%%%%%%%%%%%%%%%%%%%%%%
In  this section, we consider extending Yamaguchi--Yau's and Hosono--Konishi's results 
\cite{YY, HoKo} to $\F^{(g,h)}$. 

\subsection{The generators of polynomial ring}

Let $\theta_z=z \frac{\partial}{\partial z}$.

We define
\begin{equation}\label{eq:definitions}
\begin{split}
&
A_p=\frac{\theta_z \G}{\G}~,
\qquad
B_p=\frac{\theta_z e^{-K}}{e^{-K}}
\qquad (p=1,2,\ldots)~,
\\&
Q_p=z^{\frac{1}{2}}\theta_z \mathcal{T} 
\qquad\qquad\qquad(p=0,1,2,\ldots)~,
\\&
R_1=z^{\frac{5}{2}}\frac{e^{K}C_{zzz}}{G_{z\bar{z}}}
\overline{D}_{\bar{z}} \overline{\mathcal{T}}
~,\qquad
R_2=z^{\frac{7}{2}}e^{K} C_{zzz} \mathcal{T}.
\end{split}
\end{equation}

The generators $A_p$'s and $B_p$'s were defined in \cite{YY}. 
The new ingredients are $Q_p$'s, $R_1$ and $R_2$ 
which are necessary for incorporating 
$\triangle_{zz}$. 

Consider the polynomial ring
\begin{equation}
I=\mathbb{C}(z)[A_1,B_1,B_2,B_3,Q_0,Q_1,Q_2,Q_3,R_1,R_2]
\end{equation}
with coefficients in the field 
of rational functions $\mathbb{C}(z)$.

\begin{lemma}
1.  $A_p \in I$ $(p\geq 2)$, $B_p \in I$ $(p\geq 4)$, $Q_p \in I$ $(p\geq 4)$.
\\
2. $\theta_z  I\subseteq I$
\end{lemma}

\begin{proof}
First, notice that the logarithmic derivation $\theta_z$ 
acts as follows:
\begin{equation}\label{eq:LogDer}
\begin{split}
&\theta_z A_p=A_{p+1}-A_pA_1~,\qquad
\theta_z B_p=B_{p+1}-B_pB_1~,\\
&\theta_z Q_p=\frac{1}{2}Q_p+Q_{p+1}~,\\
&\theta_z R_1=\Big(\frac{5}{2}-A_1-B_1+
\frac{\theta_z C_{zzz}}{C_{zzz}}\Big)R_1+R_2,
\\
&\theta_z R_2=\Big(\frac{7}{2}-B_1+
\frac{\theta_z C_{zzz}}{C_{zzz}}\Big)R_2~.
\end{split}
\end{equation}

Next we show $A_2,B_4,Q_4\in I$.
By the special geometry relation (\ref{eq:special}), we have
\begin{equation}
   A_2 = -2A_1B_1+2B_1^2+2B_1-4B_2+
          \frac{\theta_z(z C_{zzz})}{zC_{zzz}}(1+A_1+2B_1)
          +h(z)~.
\end{equation}
Here $h(z)$ is determined by
comparing the
behaviour  of the RHS and the LHS at $z=0$:
\begin{equation}
 \nonumber
  h(z) = \frac{1-3\cdot 5^4 z}{1-5^5z}~.
\end{equation}
Let us write the Picard--Fuchs operator as
$\mathcal{D}=\sum_{p=0}^4H_p(z){\theta_z}^p$   
where $H_p(z)\in \mathbb{C}[z]$.
Since $\mathcal{D} e^{-K}=0$, $B_4$ satisfies
\begin{equation}
 \label{eq:B4}
    B_4=-\sum_{p=1}^3 \frac{H_p(z)}{H_4(z)} B_p
         -\frac{H_0(z)}{H_4(z)} = 0~.
\end{equation}
Moreover, since $\mathcal{T}$ satisfies (\ref{eq:EXPF}), 
\begin{equation}
 \label{eq:Q4}
   Q_4= -\sum_{p=0}^3 \frac{H_p(z)}{H_4(z)} Q_p +
         \frac{60}{2^4}z.
\end{equation}

These together with \eqref{eq:LogDer} implies that
$I$ is closed with respect to the logarithmic derivation $\theta_z$.
Moreover,
by applying $\theta_z$ recursively,
we can show that 
$A_p \in I$ $(p\geq 3)$, $B_p \in I$ $(p\geq 5)$, $Q_p \in I$ $(p\geq 5)$.
\end{proof}

%%%%%%%%%%%%%%%%%%%%%%%%%%%
\subsection{Polynomiality}
%%%%%%%%%%%%%%%%%%%%%%%%%%%
%%%%%%%%%%%%%%%%
For simplicity, we will use the notation 
\begin{equation}
V_1=A_1+2B_1+1,\qquad
V_2=B_2-B_1V_1
\end{equation}
from here on.

Since $D_z$ acts on 
$(T_{\mcpl})^m\otimes \mathcal{L}^n$
as 
\begin{equation}\nonumber 
   D_z  =\frac{1}{z}(\theta_z+mA_1+nB_1)~, 
\end{equation}
we have the following 
\begin{lemma}\label{rem:cov-der}
Let $f$ be a section of $(T_{\mcpl})^m\otimes \mathcal{L}^n$. 
Then $D_z f\in I$ if $f\in I$
and $D_z f\in\sqrt{z} I$ if $f\in \sqrt{z} I$.
\end{lemma}

\begin{lemma}\label{prop:FI}
$\F^{(g,h)}_{n}\in z^{\frac{h}{2}}~I$.
\end{lemma}

\begin{proof}

We prove the lemma by induction on $(g,h)$.

For $(g,h)=(0,0),(1,0),(0,1),(0,2)$, the lemma is true since
\begin{equation}\label{eq:induction}
\begin{split}
 \F_{3}^{(0,0)}&=\C\in I,
\\
 {\F}_1^{(1,0)}
  & = \frac{1}{2z}
     \Big[-A_1-\frac{62}{3}B_1-\frac{31}{6}-
     \frac{1}{6}\frac{\theta_z(1-5^5z)}{(1-5^5z)}\Big]~\in I,
\\
 \F^{(0,1)}_{2}&=\triangle_{zz}
   = z^{-\frac{5}{2}}[Q_2-V_1Q_1-V_2 Q_0-R_1]~\in \sqrt{z}~I,
\\ 
 \F^{(0,2)}_{1}&=\frac{1\triangle_{zz}}{2\C}-\frac{B_1}{2z} 
  +f^{(0,2)} ~ \in I.
\end{split}
\end{equation}

For $(g,h)\neq (0,0),(1,0),(0,1),(0,2)$,
assume that $\F^{(g',h')}_{n}\in z^{\frac{h}{2}}I$
holds for every $(g',h')\neq (g,h)$ such that
$g'\leq g$ and $h'\leq h$. 
Consider the contribution $F_G$ from 
a Feynman diagram $G\in \Bbb{G}(g,h)$ to $\F^{(g,h)}_{\rm{FD}}$ (\ref{eq:FG}).
The assumption of the induction implies that a vertex factor satisfies
$\widetilde{C}^{(g_v,h_v)}_{n_v;m_v}\in z^{\frac{h_v}{2}}I$.
As for edge factors, the followings hold.
From \eqref{eq:propagators0},
\begin{equation}
 \label{eq:propagators}
\begin{split}
S^{zz} = \frac{1}{z C_{zzz}}
           \Big(-A_1-2B_1-\frac{8}{5}\Big)\in I,
\quad
S^z = \frac{1}{z^2 C_{zzz}}
               \Big(B_2+3B_1+\frac{2}{25}\Big)\in I.
\end{split}
\end{equation}
By Lemma \ref{rem:cov-der}, $S$ also satisfies $S\in I$.
Similarly by \eqref{eq:induction} the terminators \eqref{eq:terminators} satisfy
$$
 \triangle^z,\triangle ~\in \sqrt{z}I.
$$
Therefore, by the condition (iii) in Definition \ref{def:graph-set},
we have $F_G\in z^{\frac{h}{2}}I$ 
and thus $\F^{(g,h)}_{\rm{FD}}\in z^{\frac{h}{2}}I$.
As to the holomorphic ambiguity $f^{(g,h)}$,  it satisfies
$f^{(g,h)}\in z^{\frac{h}{2}}\mathbb{C}(z)\subset
z^{\frac{h}{2}}~I$ by assumption \eqref{eq:hol-amb}.
Therefore $\F^{(g,h)}\in I$. 
For $n\geq 1$, $\F^{(g,h)}_n \in I$ by Lemma \ref{rem:cov-der}.
\end{proof}

\begin{definition}
\label{def:Pghn}
%%%%%%%%%%%%%%%%%%
Let  $g,h,n\geq 0$ be integers satisfying $2g-2+h+n>0$.
We define
\begin{equation}\label{eq-def:Pghn}
P^{(g,h)}_n = (z^3 \C)^{g+h-1}z^{\frac{h}{2}}
            \F^{(g,h)}_{n}~,
\quad P^{(g,h)}:=P^{(g,n)}_0. 
\end{equation}
For other values of $(g,h,n)$, we set $P^{(g,h)}_n=0$. 
\end{definition}
Lemma \ref{prop:FI} implies that
$$
P_n^{(g,h)}~\in~ I.
$$

\begin{remark}
Let $x=z^3 \C=\frac{5}{1-5^5z}$.
Consider the graded ring 
$$\Bbb{C}[x,A_1,B_1,B_2,B_3,Q_0,\ldots , Q_3, R_1,R_2],$$ 
where the grading is given by 
 deg $x=1$,
 deg $A_1=1$, 
 deg $B_p=p$ ($p=1,2,3$),
 deg $Q_p=p$ ($p=0,1,2,3$),
 deg $R_1=2$ 
 and deg $R_2=3$.
Then 
$P^{(g,h)}$ belongs to this ring
and its degree is at most $3(g+h-1)$.

\end{remark}

%%%%%%%%%%%%%%%%%%%%%%%%%%%%%%%%%%%%%%%%%%%%%%%%%%%%%%%%%%%%%%%%%%%%%%%%%%%%%%%%%%%%%%
\subsection{Rewriting the extended holomorphic anomaly
equation \eqref{eq:EHAE}}
%%%%%%%%%%%%%%%%%%
There are relations among the $\delbar$-derivatives of the generators 
\eqref{eq:definitions}. 
%%%%%%%%%%%%%%
\begin{lemma}
%%%%%%%%%%%%%%
\begin{equation}\label{eq:delbar}
\begin{split}
\partial_{\bar{z}}B_2&=V_1\partial_{\bar{z}}B_1~,
\\
\delbar B_3&=
(A_2+2A_1+3B_1+3B_2+3A_1B_1+1)
\delbar B_1~
\\ 
        &=\Big(-V_2+\frac{\theta_z (z^3\C)}{z^3\C}V_1
     +h(z)-1\Big)\delbar B_1
\\
\partial_{\bar{z}}Q_p&=0 \qquad (p=0,1,2,\ldots)~,
\\
\partial_{\bar{z}}R_2&=-R_1\partial_{\bar{z}}B_1~.
\end{split}
\end{equation}
\end{lemma}

\begin{proof}
The first and the second equations were obtained 
from (\ref{eq:special}) in \cite{YY}.
The third is trivial since $Q_p$'s do not depend on $\bar{z}$.
The calculation of  $\delbar R_2$ is as follows.
$$
\partial_{\bar{z}}R_2= z^{\alpha+1}C_{zzz}(\partial_{\bar{z}}\bar{T}+
\partial_{\bar{z}}K\cdot \bar{T})=z G_{z\bar{z}}R_1
=-R_1\partial_{\bar{z}}B_1
$$
where we have used the identity 
$\G = \partial_z\partial_{\bar{z}}K(z,\bar{z})
              = -\partial_{\bar{z}}B_1/z$.
\end{proof}

If one assumes that
$\partial_{\bar{z}}A_1$, $\partial_{\bar{z}}B_1$, 
$\partial_{\bar{z}}R_1$ are independent,
the Walcher's extended holomorphic equation 
\eqref{eq:EHAE}
is rewritten as follows.
%%%%%%%%%%%%%%%%%%%%%%%%%%%%%%
\begin{lemma}\label{prop:HAE1}
%%%%%%%%%%%%%%%%%%%%%%%%%%%%%%
The equation 
\eqref{eq:EHAE} is equivalent to 
the system of PDE's:
\begin{align}
\label{hae1}
&
\Big[-R_1\frac{\partial}{\partial R_2}
-2\frac{\partial}{\partial A_1}
+
\frac{\partial}{\partial B_1}+
V_1 \frac{\partial}{\partial B_2}
\\\nonumber
&+
\Big(-V_2+\frac{\theta_z (z^3\C)}{z^3\C}V_1+h(z)-1\Big)
\frac{\partial}{\partial B_3}
\Big]P^{(g,h)}=0~,
\\
&
\label{hae2}
\frac{\partial P^{(g,h)}}{\partial A_1}=
-\frac{1}{2}\Big(
\sum_{\begin{subarray}{c}
g_1+g_2=g,\\h_1+h_2=h
\end{subarray}}
P_1^{(g_1,h_1)}P_1^{(g_2,h_2)}+P_2^{(g-1,h)}
\Big)
+(B_1Q_0-Q_1)P_1^{(g,h-1)}~,
\\
&
\label{hae3}
\frac{\partial P^{(g,h)}}{\partial R_1}=-P_1^{(g,h-1)}
~.
\end{align}
Here the summation in \eqref{hae2} runs over
$(g_1,h_1),(g_2,h_2)$ such that $(g_i,h_i)\neq (0,0),(0,1)$.

\end{lemma}

\begin{proof}

By \eqref{eq:EHAE}, we have 
\begin{equation}\nonumber
\begin{split}
\partial_{\bar{z}}P^{(g,h)}
 &=
 \frac{1}{2}\partial_{\bar{z}}(z C_{zzz}S^{zz})\cdot
   \Big( \sum_{\begin{subarray}{c}
                g_1+g_2=g,\\h_1+h_2=h
                               \end{subarray}}
      P_1^{(g_1,h_1)}P_1^{(g_2,h_2)}+P_2^{(g-1,h)}
   \Big)
\\
    &-\partial_{\bar{z}}(z^{\frac{5}{2}}
       \C\triangle^z)\cdot P_1^{(g,h-1)}~.
\end{split}
\end{equation}
Note that, by \eqref{eq:propagators}\eqref{eq:delbar}, 
\begin{equation}\nonumber
\begin{split}
 &
   \partial_{\bar{z}}(z C_{zzz}S^{zz})=
  -(\partial_{\bar{z}}A_1+2\partial_{\bar{z}}B_1),
  \\
 &\partial_{\bar{z}}(z^{\frac{5}{2}}C_{zzz}\triangle^z)=
 -(\partial_{\bar{z}}A_1+2\partial_{\bar{z}}B_1)(-Q_1+B_1Q_0)
 +\partial_{\bar{z}}R_1~. 
\end{split}
\end{equation}

On the other hand, 
by \eqref{eq:delbar}, 
$\delbar$ in the LHS is
as follows :
\begin{equation}\nonumber
\begin{split}
\partial_{\bar{z}} &=
\partial_{\bar{z}}R_1 \frac{\partial}{\partial R_1}
+
\partial_{\bar{z}}A_1 \frac{\partial}{\partial A_1}
+\partial_{\bar{z}}B_1 \Big[
-R_1\frac{\partial}{\partial R_2}
+
\frac{\partial}{\partial B_1}+
V_1\frac{\partial}{\partial B_2}
\\& +
\Big(-V_2+\frac{\theta_z (z^3\C)}{z^3\C}V_1
     +h(z)-1\Big) \frac{\partial}{\partial B_3}
\Big].
\end{split}
\end{equation}
Inserting these and comparing 
the coefficients of 
$\partial_{\bar{z}}A_1,
\partial_{\bar{z}}B_1,
\partial_{\bar{z}}R_1
$,
one obtains Lemma \ref{prop:HAE1}.

\end{proof}

To write the equations in a more useful form,
we change the generators. 
We define
\begin{equation}\label{eq:variable-change}
\begin{split}
u &= B_1,\qquad 
v_1=V_1 +\frac{3}{5},\qquad
v_2=V_2+\frac{2}{25},
\\
v_3&=B_3-B_1\Big(-V_2+\frac{\theta_z (z^3\C)}{z^3\C}V_1+h(z)-1\Big)
 +s(z),\\
m_1&=\frac{2}{25}Q_0+\frac{3}{5}Q_1+Q_2-R_1,\\
m_2&=Q_0\Big(s(z)-\frac{2}{25}\frac{\theta_z (z^3\C)}{z^3\C}\Big)
+Q_1\Big(\frac{23}{25}-h(z)\Big)-Q_2\frac{\theta_z(z^3\C)}{z^3\C}
\\&+Q_3 -R_2-B_1R_1,
\end{split}
\end{equation}
where
\begin{equation}
s(z)=\frac{12}{25}-\frac{1}{5}h(z)
+\frac{3}{25}\frac{\theta_z (z^3\C)}{z^3\C}.
\end{equation}

Define the ring
$$
J:=\mathbb{C}(z)[u,v_1,v_2,v_3,
Q_0,Q_1,Q_2,Q_3,m_1,m_2].
$$
It is isomorphic to $I$ since
\eqref{eq:variable-change} is invertible.
Notice that $\theta_z:J\to J$ increases the 
degree in $u$ at most by 1.

Now we regard $P^{(g,h)}\in J$.
Then \eqref{hae1} implies $P^{(g,h)}$ is independent of $u$.
In turn, $P^{(g,h)}_n\in J$ has degree at most $n$ in $u$.
Following \cite[(3-4.c)]{HoKo},
let us  define $u$-independent polynomials
$Y_0,Y_1,W_0,W_1,W_2\in J$ by
\begin{equation}
\begin{split}
&
Y_0+u~Y_1=P_1^{(g,h-1)},
\\
&
W_0+u W_1+u^2W_2 =(\text{the RHS of \eqref{hae2}}).
\end{split}
\end{equation}
Then applying the change of generators 
\eqref{eq:variable-change} 
to the equations 
\eqref{hae1}\eqref{hae2}\eqref{hae3},
we obtain

\begin{theorem}\label{prop:EHAE2}
The equation \eqref{eq:EHAE} is equivalent to the following
system of PDE's for
$P^{(g,h)}\in J$ : 
\begin{equation}\label{eq:EHAE2}
\begin{split}
&\frac{\partial}{\partial u}P^{(g,h)}=0,
\\
&
\frac{\partial}{\partial m_1}P^{(g,h)}=Y_0,
\quad
\frac{\partial}{\partial m_2}P^{(g,h)}=Y_1,
\\
&\frac{\partial}{\partial v_1}P^{(g,h)}=W_0,
\quad
\frac{\partial}{\partial v_2}P^{(g,h)}=-W_1
 +\frac{\theta_z (z^3\C)}{z^3\C} W_2,
\quad
\frac{\partial}{\partial v_3}P^{(g,h)}=-W_2.
\end{split}
\end{equation}

\end{theorem}

Let us comment on the constant of integration.   
Decompose $P^{(g,h)}$ as
$$
P^{(g,h)}=\hat{P}^{(g,h)}+P^{(g,h)}|_{v_1,v_2,v_3,m_1,m_2=0}
$$
where $\hat{P}^{(g,h)}$ consists of terms 
of degree $\geq 1$ with respect to at least one of
$v_1,v_2,v_3,m_1,m_2$.
The equations (\ref{eq:EHAE2}) 
can determine $\hat{P}^{(g,h)}$, but not the second term. 
The latter is a priori a polynomial in $Q_0,Q_1,Q_2,Q_3$ with $\mathbb{C}(z)$ coefficients.
However, the choice of the new generators \eqref{eq:variable-change}
is ``good'' (cf. \cite[(3-4.d)]{HoKo}) so that
we have the following 
\begin{proposition}
$$
P^{(g,h)}|_{v_1,v_2,v_3,m_1,m_2=0} = (z^3\C)^{g+h-1}z^{\frac{h}{2}}f^{(g,h)}.
$$
\end{proposition}
\begin{proof}
We have 
\begin{equation}\nonumber
\begin{split}
S^{zz}&=-\frac{v_1}{z \C},\quad
S^z=\frac{uv_1+v_2}{z^2 \C},
\\
S&=\frac{1}{z^3 C}\Big[
-\frac{1}{2}u^2v_1-\Big(u+\frac{5^5z}{2(1-5^5z)}\Big)v_2
+\frac{v_3}{2}\Big],
\\
\triangle^z&=\frac{1}{z^{\frac{5}{2}}\C}
(-m_1+Q_1v_1+Q_0v_2),
\\
\triangle&=\frac{1}{z^{\frac{7}{2}}\C}
\Big[
um_1-m_2-uQ_0v_1-v_2\Big(
uQ_0+\frac{5^5z}{1-5^5z}Q_0+Q_1
\Big)
+Q_0v_3
\Big]~.
\end{split}
\end{equation}
Notice that every monomial term in
the propagators $S^{zz},S^z,S$
and the terminators $\triangle^z,\triangle$
contains at least one of 
$v_1,v_2,v_3,m_1,m_2$.
Therefore 
the Feynman diagram part $\F^{(g,h)}_{FD}$ 
of $\F^{(g,h)}$ has degree at least one
with respect to one of $v_1,v_2,v_3,m_1,m_2$ by \eqref{eq:FG}\eqref{eq:FD}.
This implies that the first term in the RHS of
$$
P^{(g,h)}=(z^3 \C)^{g+h-1}z^{\frac{h}{2}}\F^{(g,h)}_{FD} 
        + (z^3 \C)^{g+h-1}z^{\frac{h}{2}}f^{(g,h)}
$$
vanishes as $v_1,v_2,v_3,m_1,m_2$ go to zero.
This proves the proposition.
\end{proof}

%%%%%%%%%%%%%%%
\section{Fixing holomorphic ambiguity and $n^{(g,h)}_d$}
%%%%%%%%%%%%
Let  $\omega_0(z)$,$\omega_1(z),\omega_2(z),\omega_3(z)$ be
the following solutions to 
the Picard--Fuchs equation $\mathcal{D}\omega=0$ about $z=0$.
\begin{equation}\nonumber
\omega_i(z)=\partial_{\rho}^i \Big(
\sum_{n\geq 0}\frac{(5\rho+1)_{5n}}
{{(\rho+1)_n}^5}z^{n+\rho}\Big)\Bigg|_{\rho=0}.
\end{equation}
Let $t=\omega_1(z)/\omega_0(z)$ be the mirror map
and consider the inverse $z=z(q)$ where $q=e^t$.
Explicitly, these are
\begin{equation}\nonumber
\begin{split}
  \omega_0(z)&=1+120z+113400z^2+\cdots,
\\
  \omega_1(z)&=\omega_0(z)\log z +770z+810225z^2+\cdots,
\\
  t&=770z+717825z^2+\frac{3225308000}{3}z^3+\cdots,
\\ 
  z&=q-770q^2+171525q^3+\cdots ~.
\end{split}
\end{equation}

Let
\begin{equation}
F_A^{(g,h)}= \lim_{\bar{z}\to 0} \F^{(g,h)} \omega_0(z)^{2g+h-2},
\end{equation}
for  $(g,h)$ satisfying $2g+h-2>0$ \footnote{
For $(g,h)=(0,0),(1,0),(0,1),(0,2)$, one should consider 
\begin{equation}\nonumber
\partial_t ^n F_A^{(g,h)}
=\Big(\frac{dz}{dt}\Big)^n \lim_{\bar{z}\to 0}
 \F_n^{(g,h)}\omega_0^{2g+h-2}
\end{equation}
where $n=3,1,2,1$, respectively.
}%
.
The limit $\bar{z}\to 0$ in the RHS means 
$$
 \G\to \frac{dt}{dz},\quad 
 e^{K} \to \omega_0(z),\quad
 \triangle_{zz} \to D_z D_z \mathcal{T}.
$$
Define $n^{(g,h)}_d$ for $h>0$ 
\footnote{
For $h=0$, the BPS number $n^g_d$ is defined by \cite{GV}
\begin{equation*}
\begin{split}
\sum_{g=0}^{\infty}{g_s}^{2g-2}F_A^{(g,0)}
&= \sum_{g=0}^{\infty}
\sum_{d>0}\sum_{k>0}
n_d^{g}\frac{1}{k}\Big(2\sin\frac{kg_s}{2}\Big)^{2g-2}q^{kd}
+\text{ polynomial in $\log q$}.
\end{split}
\end{equation*}
}\label{ft:GoVa}
by the formula \cite{OV} \cite{LMV} \cite[(3.22)]{Wa2}:
\begin{equation}\label{eq:multiple-cover-formula}
\begin{split}
&\text{the terms in positive powers in $q$ of }
\sum_{g=0}^{\infty}{g_s}^{2g+h-2}F_A^{(g,h)}
\\
&= \sum_{g=0}^{\infty}
\sum_{d}\sum_{k}
n_d^{(g,h)}\frac{1}{k}\Big(2\sin\frac{kg_s}{2}\Big)^{2g+h-2}q^{\frac{kd}{2}}.
\end{split}
\end{equation}
Here the summation of $k$ is over positive odd integers
and
that of $d$ is over positive
even (resp. odd) integers when $h$ is even (resp. odd).

\begin{remark}
It is expected that
$F_A^{(g,h)}$ is the $A$-model 
topological string amplitude of genus $g$ with
$h$ boundaries for the real quintic 3-fold 
$(X,L)$, 
and that $n_{d}^{(g,h)}$ be the BPS invariants
in the class $d\in H_2(X,L;Z)$.
See \cite{Givental,Zinger} for $(g,h)=(0,0),(1,0)$
and \cite{Wa1,PaSoWa} for $(g,h)=(0,1)$.
\end{remark}

In order to fix the holomorphic ambiguity,
we put the following assumptions. 
\begin{enumerate}

\item[(i)] \label{cond2}
If $h$ is even, 
the $q$-constant term in $F_A^{(g,h)}$ vanishes
except for  $(g,h)=(0,2)$.

\item[(ii)] \label{cond1}
$n_d^{(g,h)}=0$ for $d\leq d_0$
where $d_0$ is the smallest number 
necessary to completely determine unknown parameters
in $f^{(g,h)}$.
For example, $d_0=3$ for $(g,h)=(0,3),(1,1)$,
$d_0=6$ for $(g,h)=(1,2),(0,4)$ and $d_0=9$ for 
$(g,h)=(1,3),(0,5)$.
\end{enumerate}
The numbers $n_{d}^{(g,h)}$ obtained under these assumptions
are listed in 
Tables \ref{tab:BPSnumbers1} and \ref{tab:BPSnumbers2}.

\begin{remark}
The boundary conditions proposed 
in \cite{Wa2} are
the condition (i) and the condition that
\begin{equation}\label{cond3}
n^{(g,h)}_d=0 \text{ if }  n^{2g+h-1}_{d}=0.
\end{equation}
These do not give enough equations to 
 fix the unknown parameters
of $f^{(g,h)}$, unless
$(g,h)=(0,1),(0,2),(0,3),(1,1)$.
{}For this reason we assumed 
(ii) instead of
\eqref{cond3}. 

\end{remark}

\begin{remark}
For the cases listed in Tables \ref{tab:BPSnumbers1} and 
\ref{tab:BPSnumbers2},
$n_{d}^{(g,h)}$ turn out to be integers.
However, for $(g,h)=(0,7),(1,5),(2,1)$,
the holomorphic ambiguities determined by our assumptions
do not give integral $n_{d}^{(g,h)}$'s.  
\end{remark}

\begin{remark}
As a final remark,
let us comment on the expansion about 
the conifold point $z=\frac{1}{5^5}$.
By expanding $\F^{(0,4)}$ about $z=\frac{1}{5^5}$,
we see that there is no gap condition such as the one 
found in \cite[(1.2)]{HKQ}.
On the other hand, 
if one imposes the gap condition to $\mathcal{F}^{(0,4)}$
instead of $n_6^{(0,4)}=0$,
then the integrality of $n_d^{(0,4)}$'s  does not hold.

\end{remark}
%%%%%%%%%%%%%%%%%
\begin{table}[t]

%%%%% (g,h)=(0,4) %%%%%%%%
\begin{equation}\nonumber
\begin{array}{|r|r|}\hline
d&n_d^{(0,4)}\\\hline
2& 0\\
4& 0\\
6& 0\\
8& -307669500\\
10&-1290543544800\\
12&-4192442370526500\\
14&-11974312128284645400\\
16&-31709386561589633978460\\
18& -79870219101822591783739800\\
20& -194146223749422074623095454800\\
\hline
\end{array}
~~
%%%%%(g,h)=(0,5)%%%%%
\begin{array}{|r|r|}\hline
d&n_d^{(0,5)}\\\hline
1&0\\
3&0\\
5&0\\
7&0\\
9&0\\
11&-101052180000\\
13 &-6448499064000\\
15&  2809704427965432000\\
17&19034205058652662269000\\
19&   85987169904148441092385200\\\hline
\end{array}
\end{equation}

%%%%(g,h)=(0,6)%%%%%%
\begin{equation}\nonumber
\begin{array}{|r|r|}\hline
d&n_d^{(0,6)}\\\hline
2& 0\\
4& 0\\
6& 0\\
8& 0\\
10& 0\\
12& 0\\
14& 10969992383850000\\
16& 88807052603386080000\\
18&  453871851092663617206000\\
20&  1856308715086126538509560000\\\hline
\end{array}
\end{equation}
\caption{$n_{d}^{(g,h)}$ for $(g,h)=(0,4),(0,5),(0,6)$}
\label{tab:BPSnumbers1}
\end{table}
%%%%%%%%%%%
\begin{table}
\begin{equation}\nonumber
%%%%(g,h)=(1,1)
\begin{array}{|r|r|}\hline
d& n_d^{(1,1)}\\\hline
1& 0\\
3& 0\\
5&-222535\\
7& -472460880\\
9& -970639017980\\
11& -1925950714205525\\
13& -3771152449472734885\\
15& -7341083828377813532445\\
17& -14254813486499789264497980\\
19&  -27655486644196368361422400900\\
\hline
\end{array}~~
%%%%(g,h)=(1,2)%%%%%%%
\begin{array}{|r|r|}\hline
d&n_d^{(1,2)}\\\hline
2& 0\\
4& 0\\
6& 0\\
8&-1798092240\\
10&-3910898328975\\
12&-3254492224834500\\
14&  11749281716111889000\\
16& 75858033724596666836250\\
18&  284100639663878543462155290\\
20&  881568399267730913608111758000\\
\hline
\end{array}
\end{equation}
%%%(g,h)=(1,3)%%%%%%%
\begin{equation}\nonumber
\begin{array}{|r|r|}\hline
d&n_d^{(1,3)}\\\hline
1&0\\
3&0\\
5&0\\
7&0\\
9&0\\
11& 59476704611850\\
13&376498723243912410\\
15&  1597793312432171312570\\
17& 5622302692504776557418000\\
19&  17697465511801448466779111250\\
\hline
\end{array}~~
%%%%(g,h)=(1,4)%%%%%%
\begin{array}{|r|r|}\hline
d&n_d^{(1,4)}\\\hline
2& 0\\
4&0\\
6&0\\
8&0\\
10& 0\\
12& 0\\
14& -510835096894879500\\
16&  -4625213168889849497100\\
18& -26075494174267321098602160\\
20&  -116382815077174964736448167150\\
\hline
\end{array}
\end{equation}
\caption{$n_{d}^{(g,h)}$ for $(g,h)=(1,1),(1,2),(1,3),(1,4)$}
\label{tab:BPSnumbers2}
\end{table}
%%%%%%%%%%%%%%%%%%%%%%%%%%%%%%
\appendix
\section{Examples of Feynman diagrams}
Feynman diagrams for $\mathcal{F}^{(0,3)}_{\rm{FD}}$ and $\mathcal{F}^{(1,1)}_{\rm{FD}}$
have been given in eqs. (2.109) and (2.108) of \cite{Wa2} respectively
($\# \Bbb{G}(0,3)= \# \Bbb{G}(1,1)=4$). 
For $(g,h)=(0,4)$, we have $\# \Bbb{G}(0,4)=19$. See Fig. \ref{fig:FD04}. 
It is clear that the number of Feynman diagrams grows rapidly as $g$ and $h$ increase.  
For example, one can check that $\# \Bbb{G}(0,5)=83$, 
$\# \Bbb{G}(1,2)=29$, $\# \Bbb{G}(2,1)=97$.  
\begin{figure}[h]
\psfrag{0}{$-$}
\psfrag{1}{$=$}
\psfrag{2}{$+$}
\psfrag{3}{$+\; \frac{1}{2}$}
\psfrag{4}{$+$}
\psfrag{5}{$+\; \frac{1}{2}$}
\psfrag{6}{$+\;\frac{1}{6}$}
\psfrag{7}{$+\;\frac{1}{2}$}
\psfrag{8}{$+\;\frac{1}{24}$}
\psfrag{9}{$+\;\frac{1}{6}$}
\psfrag{10}{$+\;\frac{1}{8}$}
\psfrag{11}{$+\;\frac{1}{2}$}
\psfrag{12}{$+\;\frac{1}{2}$}
\psfrag{13}{$+\;\frac{1}{2}$}
\psfrag{14}{$+\;\frac{1}{2}$}
\psfrag{15}{$+$}
\psfrag{16}{$+$}
\psfrag{17}{$+\;\frac{1}{2}$}
\psfrag{18}{$+$}
\psfrag{19}{$+\;\frac{1}{2}$}
\begin{center}
\includegraphics{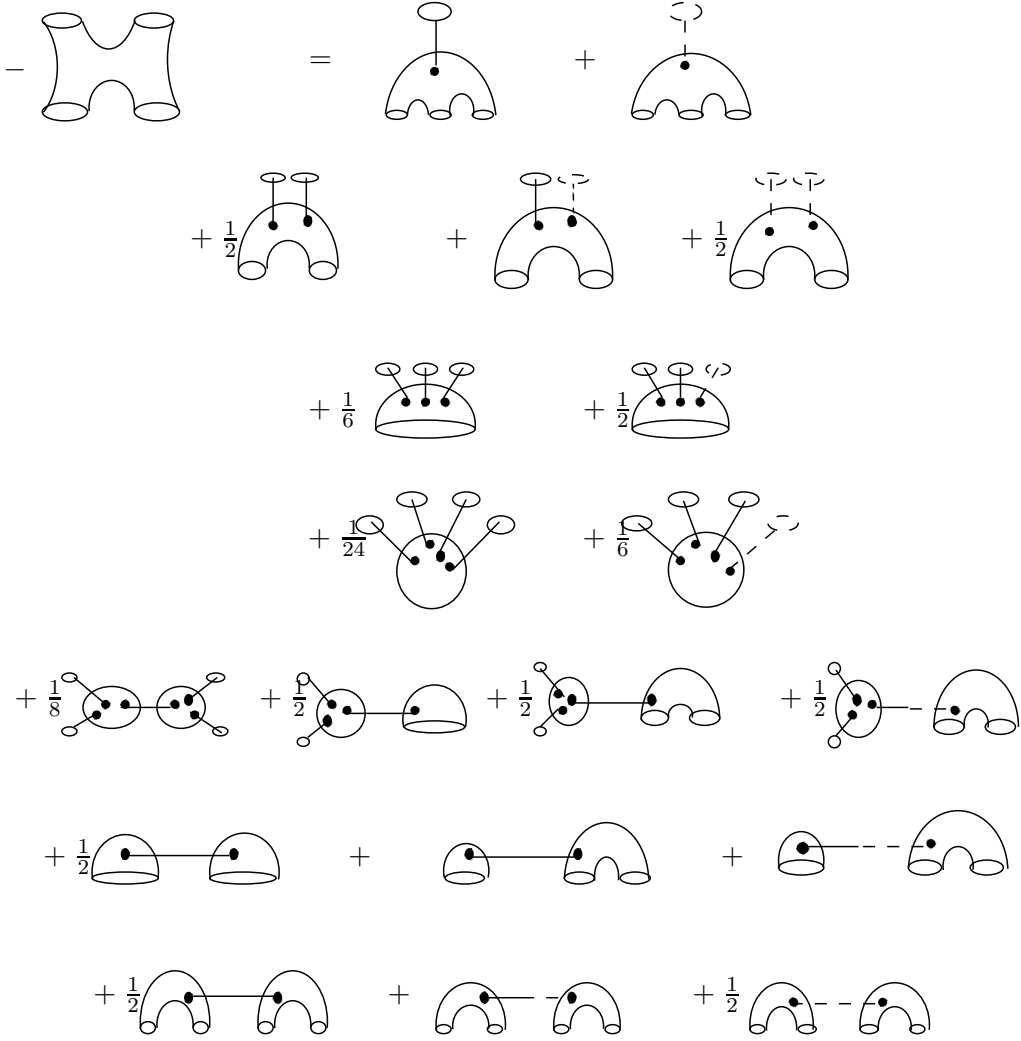}
\caption{The elements $G$ in $\Bbb{G}(0,4)$ and the orders of $A_G$. The vertices 
are expressed as bordered Riemann surfaces to visualize the labeling. }
\label{fig:FD04}
\end{center}
\end{figure}

%%%%%%%%%%%%%%%%%%%%
\section{$f^{(0,4)}$ and $P^{(0,4)}$}
\begin{equation}\nonumber
f^{(0,4)}=
\frac{ 2 - 20125\,z + 70618750\,z^2 - 86493078125\,z^3  }
  {10000\,{\left( 1 - 3125\,z \right) }^2}.
\end{equation}

\vspace*{.5cm}
{\tiny
{\allowdisplaybreaks
\begin{align*}
&P^{(0,4)}=
\frac{z^2\,( -2 + 20125\,z - 70618750\,z^2 + 86493078125\,z^3 ) }
   {80\,{( -1 + 3125\,z ) }^5} - 
  \frac{z\,( 2 - 9500\,z + 16015625\,z^2 ) \,{m_1}^2}
   {20\,{( -1 + 3125\,z ) }^3} 
\\ 
&+ 
  \frac{( -9 + 12500\,z ) \,{m_1}^4}{120\,( -1 + 3125\,z ) } + 
  \frac{75\,z^2\,( -1 + 3145\,z ) \,m_2}{4\,{( -1 + 3125\,z ) }^3} + 
  \frac{{m_1}^3\,m_2}{6} + \frac{5\,z\,{m_2}^2}{4\,( -1 + 3125\,z ) } \\
&+ 
  m_1\,( \frac{375\,z^3\,( -3 + 3125\,z ) }
      {2\,{( -1 + 3125\,z ) }^4} + 
     \frac{375\,z^2\,m_2}{2\,{( -1 + 3125\,z ) }^2} )  - 
  \frac{{Q_1}^4\,{v_1}^5}{8} + ( \frac{( -3 + 25000\,z ) \,{Q_0}^4}
      {40\,( -1 + 3125\,z ) } - \frac{{Q_0}^3\,Q_1}{6} ) \,{v_2}^4 
\\&
+ 
  {v_1}^4\,( \frac{m_1\,{Q_1}^3}{2} + 
     \frac{( -9 + 12500\,z ) \,{Q_1}^4}{120\,( -1 + 3125\,z ) } - 
     \frac{Q_0\,{Q_1}^3\,v_2}{2} )  + 
  ( \frac{25\,z^2}{8\,{( -1 + 3125\,z ) }^2} - 
     \frac{75\,z^2\,( -1 + 3145\,z ) \,Q_0}{4\,{( -1 + 3125\,z ) }^3} 
\\&
- 
     \frac{375\,z^2\,m_1\,Q_0}{2\,{( -1 + 3125\,z ) }^2} - 
     \frac{{m_1}^3\,Q_0}{6} - \frac{5\,z\,m_2\,Q_0}{2\,( -1 + 3125\,z ) } ) v_3 
+ \frac{5\,z\,{Q_0}^2\,{v_3}^2}{4\,( -1 + 3125\,z ) } 
\\
&
+ 
  {v_2}^2\,( \frac{z\,( -1 + 4750\,z + 119921875\,z^2 ) \,{Q_0}^2}
      {10\,{( -1 + 3125\,z ) }^3}
 + 
     m_1\,( \frac{-5\,z\,Q_0}{2\,( -1 + 3125\,z ) } + 
        \frac{m_2\,{Q_0}^2}{2} )  
\\&- 
     \frac{8000\,z^2\,Q_0\,Q_1}{{( -1 + 3125\,z ) }^2} + 
     \frac{5\,z\,{Q_1}^2}{4\,( -1 + 3125\,z ) } + 
     {m_1}^2\,( \frac{( -9 + 43750\,z ) \,{Q_0}^2}
         {20\,( -1 + 3125\,z ) } - \frac{Q_0\,Q_1}{2} )  - 
     \frac{m_1\,{Q_0}^3\,v_3}{2} )  
\\&
+ 
  {v_2}^3\,( \frac{5\,z\,{Q_0}^2}{4\,( -1 + 3125\,z ) } - 
     \frac{m_2\,{Q_0}^3}{6} + m_1\,( \frac{-( ( -9 + 59375\,z ) \,
             {Q_0}^3 ) }{30\,( -1 + 3125\,z ) } + \frac{{Q_0}^2\,Q_1}{2}
        )  + \frac{{Q_0}^4\,v_3}{6} )  
\\&
+ 
  {v_1}^3\,( \frac{-375\,z^2\,{Q_1}^2}{4\,{( -1 + 3125\,z ) }^2} - 
     \frac{3\,{m_1}^2\,{Q_1}^2}{4} - \frac{( -9 + 12500\,z ) \,m_1\,{Q_1}^3}
      {30\,( -1 + 3125\,z ) } - \frac{m_2\,{Q_1}^3}{6} 
\\&+ 
     ( \frac{3\,m_1\,Q_0\,{Q_1}^2}{2} + \frac{3\,Q_0\,{Q_1}^3}{10} - 
        \frac{{Q_1}^4}{6} ) \,v_2 - \frac{3\,{Q_0}^2\,{Q_1}^2\,{v_2}^2}{4} 
+ 
     \frac{Q_0\,{Q_1}^3\,v_3}{6} )  
\\&+ 
  v_2\,( \frac{81875\,z^3}{8\,{( -1 + 3125\,z ) }^3} - 
     \frac{236625\,z^3\,Q_0}{4\,{( -1 + 3125\,z ) }^3} + 
     {m_1}^2\,( \frac{5\,z}{4\,( -1 + 3125\,z ) } - \frac{m_2\,Q_0}{2} )
\\&
         + {m_1}^3\,( \frac{-3\,Q_0}{10} + \frac{Q_1}{6} )  + 
     \frac{75\,z^2\,( -1 + 3145\,z ) \,Q_1}{4\,{( -1 + 3125\,z ) }^3} + 
     m_1\,( \frac{z\,( -1 + 1625\,z ) \,Q_0}
         {5\,{( -1 + 3125\,z ) }^2} + 
        \frac{375\,z^2\,Q_1}{2\,{( -1 + 3125\,z ) }^2} ) 
\\& + 
     m_2\,( \frac{-8000\,z^2\,Q_0}{{( -1 + 3125\,z ) }^2} + 
        \frac{5\,z\,Q_1}{2\,( -1 + 3125\,z ) } )  + 
     ( \frac{8000\,z^2\,{Q_0}^2}{{( -1 + 3125\,z ) }^2} + 
        \frac{{m_1}^2\,{Q_0}^2}{2} - \frac{5\,z\,Q_0\,Q_1}{2\,( -1 + 3125\,z ) }
        ) \,v_3 )
\\&  + v_1\,
   ( \frac{-140625\,z^4}{8\,{( -1 + 3125\,z ) }^4} - \frac{{m_1}^4}{8} - 
     \frac{375\,z^3\,( -3 + 3125\,z ) \,Q_1}{2\,{( -1 + 3125\,z ) }^4} + 
     \frac{z\,( 2 - 9500\,z + 16015625\,z^2 ) \,m_1\,Q_1}
      {10\,{( -1 + 3125\,z ) }^3} 
\\&- 
     \frac{( -9 + 12500\,z ) \,{m_1}^3\,Q_1}{30\,( -1 + 3125\,z ) }
  - 
     \frac{375\,z^2\,m_2\,Q_1}{2\,{( -1 + 3125\,z ) }^2} + 
     {m_1}^2\,( \frac{-375\,z^2}{4\,{( -1 + 3125\,z ) }^2} - 
        \frac{m_2\,Q_1}{2} ) 
\\& + 
     ( \frac{m_1\,{Q_0}^3}{2} + \frac{( -9 + 59375\,z ) \,{Q_0}^3\,Q_1}
         {30\,( -1 + 3125\,z ) } - \frac{{Q_0}^2\,{Q_1}^2}{2} ) \,{v_2}^3
  - 
     \frac{{Q_0}^4\,{v_2}^4}{8} 
\\& + ( \frac{375\,z^2\,Q_0\,Q_1}
         {2\,{( -1 + 3125\,z ) }^2} + \frac{{m_1}^2\,Q_0\,Q_1}{2} ) \,v_3 + 
     v_2\,( \frac{{m_1}^3\,Q_0}{2} - 
        \frac{z\,( -1 + 1625\,z ) \,Q_0\,Q_1}{5\,{( -1 + 3125\,z ) }^2} - 
        \frac{375\,z^2\,{Q_1}^2}{2\,{( -1 + 3125\,z ) }^2} 
\\&+
        m_1\,( \frac{375\,z^2\,Q_0}{2\,{( -1 + 3125\,z ) }^2} - 
           \frac{5\,z\,Q_1}{2\,( -1 + 3125\,z ) } + m_2\,Q_0\,Q_1 )  + 
        {m_1}^2\,( \frac{9\,Q_0\,Q_1}{10} - \frac{{Q_1}^2}{2} )  - 
        m_1\,{Q_0}^2\,Q_1\,v_3 )  
\\&+ 
     {v_2}^2\,( \frac{-375\,z^2\,{Q_0}^2}{4\,{( -1 + 3125\,z ) }^2} - 
        \frac{3\,{m_1}^2\,{Q_0}^2}{4} + 
        \frac{5\,z\,Q_0\,Q_1}{2\,( -1 + 3125\,z ) } - 
        \frac{m_2\,{Q_0}^2\,Q_1}{2} 
\\&+ 
        m_1\,( \frac{-( ( -9 + 43750\,z ) \,{Q_0}^2\,Q_1 ) }
            {10\,( -1 + 3125\,z ) } + Q_0\,{Q_1}^2 )  + 
        \frac{{Q_0}^3\,Q_1\,v_3}{2} )  )  
\\& + 
  {v_1}^2\,( \frac{{m_1}^3\,Q_1}{2} - 
     \frac{z\,( 2 - 9500\,z + 16015625\,z^2 ) \,{Q_1}^2}
      {20\,{( -1 + 3125\,z ) }^3} 
 + 
     \frac{( -9 + 12500\,z ) \,{m_1}^2\,{Q_1}^2}{20\,( -1 + 3125\,z ) } 
\\& + 
     m_1\,( \frac{375\,z^2\,Q_1}{2\,{( -1 + 3125\,z ) }^2} + 
        \frac{m_2\,{Q_1}^2}{2} )  + 
     ( \frac{3\,m_1\,{Q_0}^2\,Q_1}{2} + 
        \frac{( -9 + 43750\,z ) \,{Q_0}^2\,{Q_1}^2}
         {20\,( -1 + 3125\,z ) } - \frac{Q_0\,{Q_1}^3}{2} ) \,{v_2}^2 
\\&- 
     \frac{{Q_0}^3\,Q_1\,{v_2}^3}{2} 
 - \frac{m_1\,Q_0\,{Q_1}^2\,v_3}{2} + 
     v_2\,( \frac{-375\,z^2\,Q_0\,Q_1}{2\,{( -1 + 3125\,z ) }^2} - 
        \frac{3\,{m_1}^2\,Q_0\,Q_1}{2} + 
        \frac{5\,z\,{Q_1}^2}{4\,( -1 + 3125\,z ) } 
\\&- \frac{m_2\,Q_0\,{Q_1}^2}{2} + 
        m_1\,( \frac{-9\,Q_0\,{Q_1}^2}{10} + \frac{{Q_1}^3}{2} )  + 
        \frac{{Q_0}^2\,{Q_1}^2\,v_3}{2} )  ).
\end{align*}
}
}

%%%%%%%%%%%%%%%%%%%%%%%%%%%

\end{document}